\documentclass[12pt]{amsart}
\usepackage{pgfplots}
\usepackage{float}
\usepackage[utf8]{inputenc}
\usepackage{graphicx}
\usepackage{color}
\usepackage{mdwlist}
\usepackage{amssymb}
\usepackage{subfigure}
\usepackage{relsize}
\usepackage{graphics}
\allowdisplaybreaks

\usepackage{cleveref}

\newcommand{\beq}{\begin{eqnarray*}}
\newcommand{\feq}{\end{eqnarray*}}
\newcommand{\beqn}{\begin{eqnarray}}
\newcommand{\feqn}{\end{eqnarray}}

\DeclareMathOperator{\sign}{sgn}
\newcommand{\RN}[1]{%
  \textup{\uppercase\expandafter{\romannumeral#1}}%
}

\newtheorem{theorem}{Theorem}[section]
\newtheorem{lemma}[theorem]{Lemma}
\newtheorem{corollary}[theorem]{Corollary}
\newtheorem{proposition}[theorem]{Proposition}
\theoremstyle{definition}

\theoremstyle{remark}
\newtheorem{remark}[theorem]{Remark}
\numberwithin{equation}{section}

\newcommand\numberthis{\addtocounter{equation}{1}\tag{\theequation}}

\begin{document}
\title[Critical thresholds in a hyperbolic relaxation system]
{Sharp critical thresholds in a hyperbolic system with relaxation}

\author{Manas Bhatnagar and Hailiang Liu}
\address{Department of Mathematics, Iowa State University, Ames, Iowa 50010}
\email{manasb@iastate.edu}
\email{hliu@iastate.edu} 
\begin{abstract} 
We propose and study a one-dimensional $2\times 2$ hyperbolic Eulerian system with local relaxation   from critical threshold phenomena perspective. The system features dynamic transition between strictly and weakly hyperbolic. For different classes of relaxation we identify intrinsic {\bf critical thresholds}  for initial data that distinguish global regularity and finite time blowup.  For relaxation independent of density, we estimate bounds on density in terms of velocity where the system is strictly hyperbolic.       
\end{abstract}
\maketitle

\section{Introduction}
It is a generic phenomena that homogeneous systems of quasilinear hyperbolic systems break down, i.e., the derivative of solutions become unbounded in finite time. The presence of source terms can lead to a delicate balance and persistence of global-in-time solutions for a large set of initial data. The existence of a threshold manifold on the initial phase space so that initial data on one side of the curve results in global-in-time solutions while the other side leads to solutions having shocks/concentration in finite time, is precisely the critical threshold phenomena.
For Euler-Poisson equations, it was studied for the first time in \cite{ELT01} , followed by threshold analysis on various hyperbolic balance laws, see, e.g., 
 \cite{BL19, CCZ16, LL08, LL09j,LL09, LT02, LT03, LT04, TT14, WTB12}.
 
In this work we are concerned with the critical threshold phenomena for the following Eulerian balance laws, 
\begin{subequations}
\label{hypMainwic}
\begin{align}
\label{hypMain}
\begin{aligned}
& \rho_t +(\rho v)_x=0,\; x\in \mathbb{R}, \; t>0,\\
&  u_t+ uu_x=\rho(v-u),\\
\end{aligned}
\end{align}
subject to initial density and velocity, 
\begin{align}
\label{hypMain2}
\begin{aligned}
(\rho (0,x), u(0,x))= (\rho_0 (x)\geq 0, u_0 (x)).
\end{aligned}
\end{align}
\end{subequations}
This system is closed only when $v$ is related to $\rho$ and $u$. 

The key motivation behind \eqref{hypMainwic}, as argued in \cite{BL202}, is to extend the usual mass-transport equation
\begin{align}\label{rv}
\rho_t +(\rho v)_x=0, 
\end{align}
into a class of balance laws.  Here, $v \in \mathbb{R}$ represents a mean velocity field. If $v$ is given in terms of the density variable $\rho$, then  (\ref{rv}) becomes closed.  
In this case, the system is considered to be in local equilibrium. 
However, very often $v$ depends on some extra variables in addition to the conserved density. The extra variable may be used to characterize non-equilibrium 
features of the system under consideration. Choosing a suitable non-equilibrium variable and determining its evolution equation are the fundamental tasks of irreversible thermodynamics \cite{GM62, Wa19}. 

In \cite{BL202}, the authors analyzed the system endowed with
$$
v= Q\ast u,  
$$
where $Q$ is a symmetric, nonnegative kernel of integral one. Their main result  states that there is a 
global-in-time $C^1$ solution if and only if 
$$
u_{0x}(x)+\rho_0(x)\geq 0 \quad \forall x\in\mathbb{R}.
$$ 
This sharp threshold is remarkable for a system of nonlocal hyperbolic equations.  A natural question arises that if the relaxation term $v$ is local, i.e., $v = f(\rho ,u)$, a smooth function, then whether the system 
\begin{subequations}
\label{main}
\begin{align} 
& \rho_t + (\rho f(\rho ,u))_x=0, \label{maina}\\ 
& u_t + uu_x = \rho (f(\rho ,u)-u) \label{mainb}
\end{align}
\end{subequations}
admits a similar critical threshold phenomena. A more specific question would be: under what conditions on $f(\rho, u)$ does this system allow for a precise characterization of critical threshold phenomena?

A unique feature of this system is that it may transition from being strictly hyperbolic to weakly hyperbolic (see Remark \ref{mixedhyperrem}). How to handle such a situation in perspective of critical threshold phenomena for \eqref{main} is the main task in this work.
The structure of the system \eqref{main} allows us to obtain an inhomogeneous transport equation for the quantity $e:=u_x+\rho$,  
\begin{align}\label{ee}
e_t + u e_x = -e(e-\rho). 
\end{align}
Thanks to this relation, we can compare the regularity of $e$ and $\rho$ and obtain global-in-time bounds on $\rho$ and $u_x$ simultaneously. However, due to $\rho$ and $u$ having different local propagation speeds, it is not sufficient to control $e$ alone to extend the solution globally.  This is the main difference between \eqref{main} and the system with 
$f=Q*u$ in \cite{BL202}, though 
equation (\ref{ee}) was also used in identifying the critical thresholds for the case $f=Q*u$.   Equation (\ref{ee}) is also similar to an equation derived by the authors in \cite{CCTT16} for the Euler alignment system which later lead to a string of regularity results for nonlocal Cucker-Smale alignment dynamics, \cite{DKRT17, HT17, KT18}. It would be interesting to develop critical threshold theory in multiple dimensions, for which the above control on $e$ is lacking.  

Our critical threshold analysis relies on local existence of classical solutions. However, the conventional local existence theorems seem to be inapplicable due to the system changing type from strictly to weakly hyperbolic. For the local existence including the weakly hyperbolic case, the key is to use the system augmented with the equation for $e=u_x+\rho$, which renders the coefficient matrix in the resulting system diagonal. Conventional existence/uniqueness theorems can then be applied to the augmented system, leading to a modified version of the local existence valid for all cases.  


\subsection{Related work}
Critical thresholds of weakly hyperbolic balance laws ($2\times 2$ systems) are relatively easier to obtain due to the propagation along a single characteristic field \cite{BL19, CCZ16, ELT01, LT01}. Strongly hyperbolic $2\times 2$ systems are relatively difficult to analyze from critical threshold point of view due to the presence of dynamic coupling between two different characteristic fields. One of the first works to analyze the blow up of $2\times 2$ hyperbolic conservation laws is by Lax \cite{Lax64}, where sufficient conditions for existence of smooth solutions for homogenous systems is derived. Later works \cite{LL09j, LL09, TW08}  leveraged these techniques to derive thresholds in $2\times 2$ systems with source terms. See \cite{TW08} for results on 1D Euler-Poisson systems with pressure, \cite{LL09j, LL09} for isentropic Euler systems with both pressure and relaxation, while the system in  \cite{LL09} is in Lagrangian coordinate.  
The system (\ref{main}) differs from the relaxation systems of \cite{LL09j, LL09} in two key aspects: (i) the implicit relaxation  form in (\ref{main}) allows for rich equilibria $f(\rho, \phi(\rho))$ with $\phi \in \{u: \; f(\rho, u)=u\}$, instead of the usual explicit relaxation form as in \cite{LL09j, LL09}; (ii)  (\ref{main}) can transition from being strongly hyperbolic to weakly hyperbolic (see Remark \ref{mixedhyperrem} and note thereafter), but systems in \cite{LL09j, LL09} are strictly hyperbolic. 
More importantly,  in these works, due to presence of two characteristic fields, one obtains in general only upper and lower thresholds. For the class of systems  (\ref{main}) with $f=f(u)$, we distinguish between strictly and weakly hyperbolic cases, and  for each case we obtain sharp critical threshold results. For general $f=f(\rho,u)$, we give sufficient conditions for global solution as well as finite time breakdown.
 
The rest of the paper is organized as follows:  Section \ref{mainresults} contains the main results. It has two subsections. The first one is devoted to local existence theorems owing to which the a priori analysis is carried out and the other contains the critical threshold results. Section \ref{frho0} contains the proofs of the main results for the case where $f$ depends on velocity only. Section \ref{genf} contains the same for general $f$. And the appendix contains the proof to Theorem \ref{local2}. 

\textbf{Notation:} $C^k_b(\mathbb{R})$ for $k=0,1,\ldots$ is the set of bounded functions which are $k$ times continuously differentiable with all the derivatives upto $k$-th order being uniformly bounded. For a function $g:\mathbb{R}\to\mathbb{R}$,
$$
||g||_\infty := \sup_{x\in\mathbb{R}} |g(x)|.
$$

\section{Main Results}
\label{mainresults}
\subsection{Local existence results}
We reformulate \eqref{main} as 
\begin{align}	
\label{matsysgen}
\left[
\begin{array}{c}
\rho\\
u
\end{array}
\right]_t +
\left[
\begin{array}{cc}
\rho f_\rho + f & \rho f_u \\
0 & u
\end{array}
\right]
\left[
\begin{array}{c}
\rho\\
u
\end{array}
\right]_x = 
\left[
\begin{array}{c}
0\\
\rho(f-u)
\end{array}
\right].
\end{align}
This is a hyperbolic system since  the two eigenvalues of the coefficient matrix,  
$$\lambda_1 =\rho f_\rho + f  \quad \text{and} \quad  \lambda_2 = u,$$
are real. Local existence of classical solutions of strictly hyperbolic systems is well known in the literature, see  \cite[Theorem 7.7.1]{Da16}. 
We state the relevant theorem for our case here.
\begin{theorem}
\label{local1}
Let $f$ be smooth such that the system \eqref{main} is strictly hyperbolic. Let $0\leq\rho_0(x)$ and $u_0(x)$ be functions in $C_b^1(\mathbb{R})$. 
Then,  there exists $T>0$ such that the solution  
$$
\rho,u \in C^1((0,T)\times\mathbb{R}).
$$ 
Moreover, the life span $T$ can be extended as long as 
$$
\|\partial_x (\rho, u)(t, \cdot)\|_\infty<\infty 
$$
and $\|(\rho, u)(t, \cdot)\|_\infty$ is bounded. 
\end{theorem}
\begin{remark} 
\label{mixedhyperrem}
In the case $f=f(u)$,  it can be shown that the set 
$$
\Sigma=\{u:\; u=f(u)\}, 
$$ 
is invariant under the system \eqref{main}. Hence for initial data with $u_0(x) \not\in \Sigma$ 
for all $x \in \mathbb{R}$, we have $u(t, x)\not\in \Sigma$ 
for $x \in \mathbb{R}$ and $t>0$. That is, system \eqref{main} is strictly hyperbolic for all $t>0$ for solutions under consideration if $u_0(x) \not\in \Sigma$.
\end{remark} 

For general $f$, the set 
$$
\Sigma_1 =\{(\rho, u):\; \lambda_1=\lambda_2\} 
$$
is not necessarily invariant. That is, even $(\rho_0, u_0)(x) \not\in \Sigma_1$  for all $x \in \mathbb{R}$,  
it is likely that  $(\rho(t, x), u(t, x)) \in \Sigma_1$ 
at some $t>0$. In other words, the system may transition from being strictly hyperbolic to weakly hyperbolic at some $t>0$. 
A modified version of the local existence for weakly hyperbolic case is thus stated in the following theorem.   
\begin{theorem}
\label{local2}
Consider the system \eqref{main} with a smooth function $f=f(\rho, u)$. 
Let $(\rho_0\geq 0 ,u_0)\in C_b^1(\mathbb{R})\times C_b^{2}(\mathbb{R}) $.
Then, there exists $T>0$ such that the solution  
$$
\rho, u \in C^1((0,T)\times\mathbb{R}).
$$ 
In addition, 
$$
u_x + \rho \in C^1((0,T)\times\mathbb{R}).
$$
Moreover, the life span $T$ can be extended as long as
$$
\|\partial_x (\rho, u)(t,\cdot)\|_\infty<\infty 
$$
and $\|(\rho, u)(t, \cdot)\|_\infty$ is bounded. 
\end{theorem}

\begin{remark} 
\label{strictweak1}
A special case for such situation is when $f(\rho,u)=u$. Then the system \eqref{main} reduces to the pressureless Euler system,
$$
\rho_t +(\rho u)_x=0, \quad  u_t+ uu_x=0,
$$ 
which is a weakly hyperbolic system. The method of characteristics allows for a precise solution of form 
$$
\rho(t, x)=\frac{\rho_0(\alpha)}{1+u_{0x}(\alpha)t}, \; u_x(t, x)=\frac{u_{0x}(\alpha)}{1+u_{0x}(\alpha)t}
$$
along $x=\alpha +u_0(\alpha)t$ for any $\alpha \in \mathbb{R}$. One can verify that when shock forms,  $u_x \to -\infty$, and simultaneously $\rho \to \infty$. In addition, we observe that in order for $\rho$ to be in $C^1$, we need $u_0\in C^2$ along with $\rho_0\in C^1$.  
\end{remark}

Note that the difference between the two local existence theorems is essentially in the smoothness required for $u_0$. This change arises due to the assumption of strict hyperbolicity in Theorem \ref{local1} which is absent in Theorem \ref{local2}. For existence of classical solutions to strictly hyperbolic systems such as the isentropic Euler system, we require the same degree of smoothness (at least $C^1$) for $\rho_0$ and $u_0$. $\rho,u$ remain bounded for all times and breakdown of smooth solution occurs when $|\rho_x|\to\infty$ or $|u_x|\to\infty$ as stated in Theorem \ref{local1}. However, for weakly hyperbolic systems such as the pressureless Euler system (see Remark \ref{strictweak1}), $-u_x$ and $\rho$ blow up simultaneously in the event of breakdown of classical solution. This physical property of weakly hyperbolic systems requires the assumption of an extra degree of smoothness in $u_0$ compared to $\rho_0$. 

To our knowledge, there seems no local existence theorems which give existence of such mixed type (see Remark \ref{mixedhyperrem}) and asymmetric systems, \eqref{matsysgen}. Therefore, we present a proof of Thoerem \ref{local2} in the Appendix.
\subsection{Threshold results}
We begin by stating results for a special form of $f$.

\begin{theorem}
\label{thm1}Let $f $ be a smooth function depending on $u$ only, i.e., $f_\rho = 0$. Consider the system \eqref{main} with initial conditions $ (\rho_0\geq 0 ,u_0) \in C_b^1(\mathbb{R})\times C_b^1(\mathbb{R})$ with $ \inf|f(u_0) - u_0 |>0$.  If $f_u \leq 0$ for solution $u$ under consideration, then 
\begin{enumerate}
\item
\textbf{Bounds on $u$ and $\rho$:}
$u(t,\cdot)$ is uniformly bounded with bounds as in \eqref{usingleroot} and satisfies $|f(u(t,\cdot))-u(t,\cdot)|>0$ for $t>0$. And
$$
\rho(t,x)\leq \frac{\sup\rho_0|f(u_0)-u_0|}{e^{\int_{u_0(x)}^{u(t,x)}\frac{d\xi}{f(\xi)-\xi}}|f(u(t,x))-u(t,x)|}.
$$
\item
\textbf{Global solution:} If
$$
u_{0x}(x)+\rho_0(x)\geq 0,\quad \forall x\in\mathbb{R},
$$ 
then there exists a global classical solution $\rho,u\in C^1((0,\infty)\times\mathbb{R})$ to \eqref{main}.
Moreover, $\rho, u_x$ are uniformly bounded  with 
$$
0 \leq \rho(t,x) \leq M, \quad 0 \leq u_x(t,x) + \rho(t,x)  \leq M, \qquad \forall t>0,x\in\mathbb{R},
$$
where $M = \max\{ \sup \rho_0 ,\sup (u_{0x}+\rho_0) \}$. 
\item
\textbf{Finite time breakdown:} If $\exists x_0\in\mathbb{R}$ such that
$$
u_{0x}(x_0)+ \rho_0(x_0)<0,
$$
then $\lim_{t\to t_c} |\rho_x| =\infty$ or $\lim_{t\to t_c} |u_x| = \infty$ for some $t_c>0$.
\end{enumerate}
\end{theorem}

\begin{theorem}
\label{thm2}Let $f $ be a smooth function depending on $u$ only, i.e., $f_\rho = 0$. Consider the system \eqref{main} subject to initial conditions, $
(\rho_0\geq 0 ,u_0) \in C^1_b(\mathbb{R})\times C^2_b(\mathbb{R}) $.
 If $f_u \leq 0$ for the solution $u$ of consideration, then 
\begin{enumerate}
\item
\textbf{Global Solution:} If
$$
u_{0x}(x)+\rho_0(x)\geq 0,\quad \forall x\in\mathbb{R},
$$
then there exists a global solution 
$$
\rho, u\in C^1((0,\infty)\times \mathbb{R}) 
$$ 
to system \eqref{main}. 
Moreover, $u,\rho, u_x$ are uniformly bounded with bounds on $u$ as in \eqref{usingleroot} and,
$$
0 \leq \rho(t,x) \leq M, \quad 0 \leq u_x(t,x) + \rho(t,x) \leq M, \qquad \forall t>0,x\in\mathbb{R},
$$
where $M = \max\{ \sup \rho_0 ,\sup (u_{0x}+\rho_0) \}$. And we have the following,
$$
||\rho_x(t,\cdot)||_\infty \leq (1+||\rho_{0x}||_\infty+||u_{0xx}||_\infty+||\rho_0||_\infty  ) e^{12(||f||_{C^2}+1)\max\{M^3,1\}t},\quad t>0.
$$
\item
\textbf{Finite time breakdown:} If $\exists x_0\in\mathbb{R}$ such that
$$
u_{0x}(x_0)+ \rho_0(x_0)<0,
$$
then $\lim_{t\to t_c} u_x = -\infty$ for some $t_c>0$ at the rate of $O\left(\frac{1}{|t-t_c|}\right)$ or faster.
\end{enumerate}
\end{theorem}

\begin{remark}
\label{strictweak2}
The condition $\inf |f(u_0) - u_0|$ in Theorem \ref{thm1} makes the system \eqref{main} strictly hyperbolic. The key difference between the two theorems is that for Theorem \ref{thm1}, $\rho,u$ are bounded for all times as long as they are so initially. This is not the case with the weakly hyperbolic system (Theorem \ref{thm2}) where there might be density concentration, i.e., $\rho$ tends to $\infty$ in finite time along with $-u_x$.
\end{remark}

Next, we state results for more general $f$.
\begin{theorem}
\label{thm3}Let $f = f(\rho,u) $ be a smooth function of its variables. Consider the system \eqref{main} with initial conditions $(\rho_0\geq 0 ,u_0) \in C_b^1(\mathbb{R})\times C_b^2(\mathbb{R})$. If $f_u \leq 0$ for the solutions under consideration, then 
\begin{enumerate}
\item
\textbf{Bounds on $u$:} There exists a smooth function $\phi:\mathbb{R}^+\to\mathbb{R}$ such that,
$$
\min\left\{ \inf_{\mathbb{R}} u_0, \min \phi(\rho) \right\}\leq u(t,\cdot )\leq \max\left\{ \sup_{\mathbb{R}} u_0, \max\phi(\rho) \right\},
$$
for as long as $\rho \geq 0$ is bounded.
\item
\textbf{Bounds on $\rho,u_x$:} If
$$
u_{0x}(x)+\rho_0(x)\geq 0,\quad \forall x\in\mathbb{R},
$$
then $\rho, u_x$ are uniformly bounded  with 
$$
0 \leq \rho(t,x) \leq M, \quad 0 \leq u_x(t,x)+ \rho(t,x)   \leq M, \qquad \forall t>0,x\in\mathbb{R}.
$$
where $M = \max\{ \sup \rho_0 ,\sup (u_{0x}+\rho_0) \}$. 
\item
\textbf{Global solution:} If in addition to $u_{0x}+\rho_0\geq 0$,
\begin{itemize}
    \item $(\rho f)_{\rho\rho}\geq 0$, $f_{uu}\leq 0$ along with 
    $$
    \rho_{0x}(x)\geq 0,\quad  u_{0xx}(x)+\rho_{0x}(x)\geq 0,\ \forall x\in\mathbb{R},
    $$ 
    \begin{center} OR \end{center}
    \item $(\rho f)_{\rho\rho}\leq 0$, $f_{uu}\geq 0$ along with 
    $$
    \rho_{0x}(x)\leq 0,\quad u_{0xx}(x)+\rho_{0x}(x)\leq  0,\ \forall x\in\mathbb{R},
    $$
\end{itemize}
then there exists a global solution $\rho, u\in C^1((0,\infty)\times\mathbb{R})$ to system \eqref{main}. In addition,
$$
u_x + \rho \in C^1((0,\infty)\times\mathbb{R}).
$$
\item
\textbf{Finite time breakdown:} If $\exists x_0\in\mathbb{R}$ such that
$$
u_{0x}(x_0)+ \rho_0(x_0)<0,
$$
then $\lim_{t\to t_c} u_x = -\infty$ at the rate of $O\left(\frac{1}{|t-t_c|}\right)$ or faster, for some $t_c>0$.
\end{enumerate}
\end{theorem}
\begin{remark}
It is worth noting that the additional conditions for global existence in assertion (3) of Theorem \ref{thm3} are not sharp, and we do not know if they are also necessary for global existence.
\end{remark}

Since the solutions persist for as long as one can place a priori bounds on the solutions and their first order derivatives, our analysis is carried out on the already existing classical local solutions. 

For strictly hyperbolic systems, we have the technique of Riemann invariants whereas weakly hyperbolic systems can be reduced to a system of ODEs along a single characteristic path. Owing to the structure of \eqref{matsysgen}, we will blend the two techniques to obtain results on critical thresholds. 


\section{Proofs of results for $f = f(u)$ case}
\label{frho0}
In the case that  $f$ depends on velocity only, the system \eqref{matsysgen} becomes
\begin{align}	
\label{matsys}
\left[
\begin{array}{c}
\rho\\
u
\end{array}
\right]_t +
\left[
\begin{array}{cc}
 f & \rho f_u \\
0 & u
\end{array}
\right]
\left[
\begin{array}{c}
\rho\\
u
\end{array}
\right]_x = 
\left[
\begin{array}{c}
0\\
\rho(f-u)
\end{array}
\right],
\end{align}
with $\lambda_1 = f(u)$ and $\lambda_2 = u$.

We first derive bounds on $u$ under some structural conditions on $f$. 
We assume that the set $\Sigma := \{ u: f(u)=u\} $ is nonempty, finite and its two endpoints are stable critical points, i.e., $f_u<1$, then $u$ can be shown uniformly bounded as long as $u_0(x)$ is. More precisely, we state the following result.
\begin{proposition}[Bounds on $u$]
\label{ubound}
Let $u_1^\ast$ and $u_N^\ast$ be the smallest and largest elements of $\Sigma$ respectively. Then as long as $\rho$ in \eqref{matsys} is bounded, we have the following:\\
If $f_u (u_1^\ast)<1$ and $f_u(u_N^\ast)<1$, then 
$$
\min\{\inf u_0, u_1^\ast\}\leq u(t,\cdot)\leq \max\{\sup u_0,u_N^\ast\} .
$$
In particular, if $\Sigma$ has exactly one element, $u^\ast$ and $f_u(u^\ast)<1$ then
\begin{align}
\label{usingleroot}
\min\{\inf u_0, u^\ast\}\leq u(t,\cdot)\leq \max\{\sup u_0,u^\ast\}. 
\end{align}
\end{proposition} 

\begin{remark}
The proposition states bounds on $u$ if $\Sigma$ has finitely many elements. However, if $\Sigma$ has infinitely many elements, then to bound $u$, we do not require the assumption of stability of the endpoints. This fact will be evident in the proof of the Proposition itself. 
For example, if $f(u) = u + \sin u$ then $\Sigma = \{\ldots, -\pi,0,\pi,\ldots \}$. In this case, $u$ will be uniformly bounded and lie in between two distinct elements of $\Sigma$ for all time. 
\end{remark}

For $\rho \geq 0$ bounded, the bound of $u$ can be obtained by considering the velocity equation \eqref{matsys} alone. By method of characteristics, along the path 
\begin{align}
\label{chpathu}
& \frac{dX}{dt} = u(t,X),\quad X(0) = x_0\in\mathbb{R}, 
\end{align}
the resulting ODE 
\begin{align}
\label{ch1}
\frac{d}{dt}u(t,X(t)) = \rho (t,X(t)) (f(u(t,X(t))) - u(t,X(t)) )
\end{align}
will serve as the key to proving the Proposition. Hereafter, we will write $X(t;x_0)$ as $X$ for the sake of simplicity.

\textit{Proof of Proposition \ref{ubound}:}
Firstly, note that elements of $\Sigma$ are critical points of \eqref{ch1}. Hence, by uniqueness of solution of ODEs, for any initial point $x_0$ with $u_{1}^\ast \leq u_0(x_0) \leq u_N^\ast$ then,
\begin{align}
\label{uproof1}
& u_{1}^\ast\leq u(t,X)\leq u_N^\ast,\ t\geq 0.
\end{align}
Now if the other condition holds, i.e., $u_0(x_0)>u_N^\ast$ (similar argument applies for $u_0(x_0)<u_1^\ast$ to get the lower bound) then 
since $u_N^\ast$ is the largest zero of $f(u)-u$ and $f_u (u_N^\ast )<1$ so, $f(u_0(x_0))-u_0(x_0) <0$. See Figure \ref{fig1}. 

\begin{figure}[h!]
\centering
\subfigure{\includegraphics[width=0.4\linewidth]{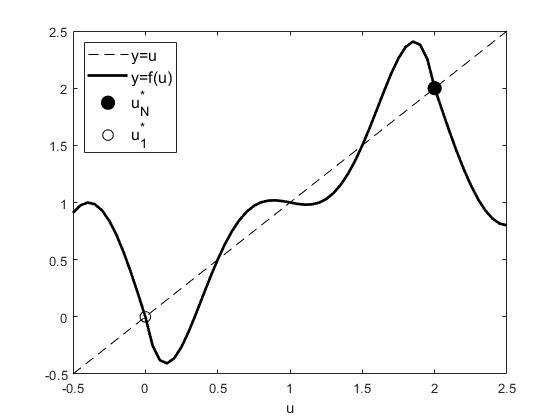}}
\caption{Terminal roots of $f(u)-u=0$.}
\label{fig1}
\end{figure}

And since $u_N^\ast$ is a critical point, we have
$$
f(u(t,X))-u(t,X) < 0, \ t\geq 0.  
$$
Because if not, then $f(u(t,X)) = u(t,X) = u_N^\ast$ for some $t>0$ which is a contradiction. 
And hence, for such $x_0$, we have from \eqref{ch1} that
\begin{align*}
& \frac{d}{dt}u(t,X) = \rho (t,X) (f(u(t,X)) - u(t,X) ) \leq 0.
\end{align*}
Therefore, $u(t,X)\leq u_0(x_0)$. Combining this with \eqref{uproof1} and collecting all characteristics, we obtain
$$
u(t,\cdot) \leq\max\{ \sup u_0, u_N^\ast\}.
$$
\qed


Next, we bound $\rho$. We motivate ourselves by the fact that $\rho$ behaves differently for strictly and weakly hyperbolic systems. In the former, there is generally no density concentration, however, that is not the case for the latter. Therefore, the bound on $\rho$ we derive depends on whether $u$ will cross a point in $\Sigma$. The following Lemma demarcates the two cases of $u$ crossing a point in $\Sigma$ or not based on the initial conditions.
\begin{lemma}
\label{invreg}
If $ |f(u_0(x))-u_0(x)|>0$ for all $x\in\mathbb{R}$, then 
$$
 |f(u(t,\cdot)) - u(t,\cdot)| > 0,\quad t\geq 0.
$$
In particular,
$$ 
\sign (f(u_0) - u_0 ) = \sign (f(u(t,\cdot)) - u(t,\cdot) ),\quad  t\geq 0.
$$
\end{lemma}
\begin{proof}
Without loss of generality assume $ f(u_0(x))-u_0(x)<0$ for all $x\in\mathbb{R}$. For a fixed $x_0\in\mathbb{R}$, set $G(t) :=f (u(t,X)) - u(t,X)$. Using \eqref{ch1}, we have
\begin{align*}
 \frac{dG}{dt} & = (f_u(u(t,X)) - 1)\frac{d}{dt}u(t,X)\\
& =(f_u(u(t,X)) - 1)  \rho (t,X) G,
\end{align*} 
which in turn gives $G(t) = G(0)e^{\int_0^t \rho(\tau, X(\tau )) \left(f_u (u(\tau, X(\tau))-1\right)d\tau}$. Therefore, 
$$
\sign (G(t)) = \sign (G(0)),
$$ 
along the path $(t,X)$ with initial point $(0,x_0)$. This argument applies for all $x_0\in \mathbb{R}$ which completes the proof. 
\end{proof}

Owing to Lemma \ref{invreg}, we have that $\{u: f(u)=u\}$ is an invariant set. Hence 
for initial data satisfying $|f(u_0(x))-u_0(x)|>0$ for all $x\in\mathbb{R}$, the system \eqref{matsys} is strictly hyperbolic for all $t>0$.   
We will obtain bounds on $\rho$ in terms of $u$ using Riemann invariant. From \eqref{matsys}, we have that the Riemann invariant corresponding to $\lambda_2$ is $u$ trivially. Let the Riemann invariant corresponding to $\lambda_1$ be $R=R(\rho ,u)$ such that  
\begin{align*}
& R_t + \lambda_1 R_x = \rho R_u (f-u).
\end{align*}
It suffices to find an appropriate integrating factor $\phi$  such that 
\begin{align*}
& [R_\rho\quad R_u] = \phi [ f-u \quad \rho f_u ].
\end{align*} 
The right hand side being the left eigenvector of coefficient matrix in \eqref{matsys} corresponding to $\lambda_1$. 
If we assume $\phi = \phi(u)$, then setting $R_{\rho u} = R_{u\rho}$, we obtain,
$$
\phi_u (f-u) + \phi(f_u -1) = \phi f_u,
$$
which results in
$$
\frac{d\phi}{\phi} = \frac{du}{f-u}.
$$
By integrating, one obtains $\phi(u) = e^{\int_{u_0}^u \frac{d\xi}{f(\xi)-\xi}}$ is a valid integrating factor. Here, the requirement  of strict hyperbolicity is used. It ensures that $\phi$, and in turn $R$, is well defined. Consequently,
\begin{align}
\label{R1}
R= \rho \phi (f-u) = \rho e^{\int_{u_0}^u \frac{d\xi}{f(\xi)-\xi}} (f-u).
\end{align}
The complete invariant system is
\begin{subequations}
\label{invsys}
\begin{align}
& R_t + f R_x = \rho^2 \phi f_u (f-u) = \rho Rf_u, \label{invsys1} \\
& u_t + u u_x = \rho (f-u). \label{invsys2}
\end{align}
\end{subequations}
For later use we set $R_0(x):=R(\rho_0(x),u_0(x))$ as initial data for $R$. 

We state bounds on $\rho$ in the following proposition.
\begin{proposition}
\label{rhobound}
Let $|f(u_0(x))-u_0(x)|>0$ for all $x\in\mathbb{R}$ and $f_u\leq 0$. Then
$$
\rho(t,x)\leq \frac{\sup\rho_0|f(u_0)-u_0|}{e^{\int_{u_0(x)}^{u(t,x)}\frac{d\xi}{f(\xi)-\xi}}|f(u(t,x))-u(t,x)|}.
$$
\end{proposition}
\begin{remark}
\label{rhoasymp}
Note that the right hand side blows up as $u\to u^\ast$ for $u^\ast\in\Sigma$. This shows that the region where $\rho$ is bounded is asymptotic to the line $u=u^\ast$ for $u^\ast\in\Sigma$. See Figure
\ref{fig2}.
\end{remark}

\begin{figure}[h!] 
\centering
\subfigure{\includegraphics[width=0.4\linewidth]{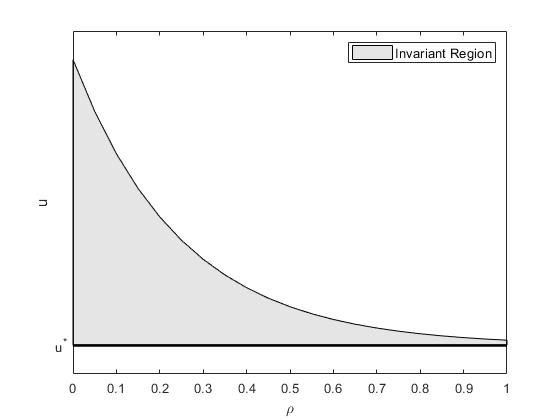}}
\caption{Asymptotic invariant region.}
\label{fig2}
\end{figure}

\textit{Proof of Proposition \ref{rhobound}:}
Without loss of generality we assume $f(u_0)-u_0<0$ for all $x\in\mathbb{R}$. Hence $R_0 (x)<0$ for all $x\in\mathbb{R}$. Fix an $\epsilon >0$ and set $\tilde R(t,x) : = R(\rho(t,x),u(t,x))+\epsilon t$. We will first show that if for some $\alpha>0$, $R_0 >-\alpha$, then 
\begin{align}
\label{tempo1}
\tilde R> -\alpha,\quad for\ t>0.
\end{align}
By way of contradiction, we assume that for the first time, $t_*$, $\tilde R(t_*,x_*) = -\alpha$ for some $x_*\in\mathbb{R}$. Therefore, $\tilde R_t(t_*,x_*)\leq 0 $. Also note that at this point $\tilde R_x = 0$ because if not then there exists some $x_1$ in the neighborhood such that $\tilde R (t_*,x_1) <-\alpha $, which contradicts that this is the first time of violation. From \eqref{invsys1}, we obtain
$$
\tilde R_t + f\tilde R_x = \rho\tilde Rf_u + \epsilon (1-\rho f_u t).
$$
At $(t_*,x_*)$,
\begin{align*}
\tilde R_t & = -\rho(t_*,x_*)\alpha f_u(u(t_*,x_*)) +  \epsilon \left[ 1-\rho(t_*,x_*) f_u (u(t_*,x_*)) t_*\right] \geq \epsilon.
\end{align*}
This is a contradiction. This proves \eqref{tempo1}. Plugging back the expression for $\tilde R$ in \eqref{tempo1}, we have that if for some $\alpha >0$, $R_0>-\alpha$, then
$$
R > -\alpha -\epsilon t, \quad t> 0.
$$
Since $\epsilon$ is arbitrary, we have $R\geq-\alpha$ is an invariant region in the $\rho-u$ plane. Consequently, we obtain
\begin{align}
\label{Rinv}
& R\geq \inf R_0.
\end{align}
Plugging \eqref{R1} above finishes the proof.\qed  

\begin{remark}
We point out that by making the assumption $f_u \leq 0$, we impose that $\Sigma$ has at most one element in the domain of consideration. From our earlier assumption of nonemptiness, we implicitly impose that $f(u) - u$ has exactly one zero in the domain of consideration.
\end{remark}


Proposition \ref{rhobound} assumes that the initial data does not cross $\{u: f(u) = u\}$. In such a case, 
one could bound $u_x$ along with $R_x$ (derivatives of Riemann invariants). 
For general initial data $u_0$ that might cross $\{u: f(u) = u \}$,   
the usual existing technique of bounding derivatives of solutions ($\rho ,u$) using Riemann invariants
is no longer applicable.  We will now show that a bound on density and $u_x$ can be obtained together through an argument by taking advantage of the structure of \eqref{matsys}. 
 Moreover, the bounds obtained are uniform. 

We proceed to take derivative of second equation in \eqref{matsys} with respect to $x$ to  obtain,
\begin{align*}
& u_{xt} + u u_{xx} = -u_x^2 + (\rho f)_x - u\rho_x - u_x\rho.  
\end{align*}
This when combined with the $\rho$ equation in \eqref{matsys}, i.e., $(\rho f)_x=-\rho_t $, 
results in
\begin{align*}
& (u_{x}+ \rho)_t + u (u_{x}+\rho)_x = -u_x(u_x + \rho).  
\end{align*}
Set 
$$
e(t, x) :=u_x (t,x) + \rho(t,x)
$$ 
with $e_0(x) = u_{0x}(x) + \rho_0(x)$, 
we thus obtain the following coupled system,
\begin{subequations}
\label{coupleerho}
\begin{align}
& \rho_t + f\rho_x = f_u\rho(\rho - e), \label{coupleerho1}\\
& e_t + u e_x = -e(e-\rho). \label{coupleerho2}
\end{align}	
\end{subequations}
This coupled system will enable us to prove the following result.
\begin{proposition}
\label{erhobound1}
Let $f_u\leq 0$. If $\rho_0 (x), e_0(x)\geq 0$ for all $x\in\mathbb{R}$  with $M = \max\{ \sup \rho_0 ,\sup e_0 \}$, then $\rho, e\in [0, M]$ for all further times.  \\
 Moreover, if $e_0(x^\ast ) <0 $ for some $x^\ast$, then $\exists x_c, t_c>0$ such that $\lim_{t\to t_c^-} e(t,x_c) = -\infty$. 
\end{proposition}
\begin{remark} Even though we allow more general initial data and obtain uniform bounds on $\rho$, we pay a price by involving the derivative of $u$. This is actually expected due to the system losing its strict hyperbolicity.
\end{remark}
\begin{proof}
Taking note of \eqref{coupleerho1}, we define another characteristic path, $(t,Y(t))$ with $Y$ as
\begin{align}
\label{chpathf}
& \frac{dY}{dt} = f(u(t,Y)),\quad Y(0) = y_0.
\end{align}
We can rewrite equations in \eqref{coupleerho}  as ODEs on the described characteristic paths $(t,Y)$ and $(t,X)$ defined by \eqref{chpathu}:
\begin{subequations}
\label{rhoe1}
\begin{align}
& \frac{d}{dt}\rho(t,Y) = f_u (u(t,Y)) \rho(t,Y)(\rho(t,Y) - e(t,Y)), \label{rhoe1a}\\
& \frac{d}{dt}e(t,X) = -e(t,X)(e(t,X)-\rho(t,X)). \label{rhoe1b}
\end{align}	
\end{subequations}
Note that $0$ is a critical point to \eqref{rhoe1a}. Therefore, if $\rho(0,x)\geq 0$, then $\rho(t,\cdot)\geq 0$ for all time. Same holds for $e(t,\cdot)$ since $0$ is a critical point of \eqref{rhoe1b} as well. 

Next, by definition of $M$, $\rho(0,x)\in [0,M]$ and $e(0,x)\in [0,M]$ for all $x\in\mathbb{R}$. We will show that $\rho(t,\cdot )$ and $e(t,\cdot )$ remains in $[0,M]$.
From \eqref{rhoe1}, we have that whenever $\rho(t,\cdot), e(t,\cdot)\in [0,M]$, then 
\begin{align*}
& \frac{d}{dt}\rho(t,Y) \leq f_u \rho(t,Y) (\rho(t,Y)-M),\\
& \frac{d}{dt}e(t,X) \leq -e(t,X)(e(t,X)-M).
\end{align*}
These inequalities hold for all points at time $t$ on both characteristic paths. Consequently, from comparison, $\rho(t,\cdot),e(t,\cdot)\in[0,M]$ for any future time.

For the second part assume $e_0(x^\ast)<0$ for some $x^\ast$. In \eqref{rhoe1b}, consider $X$ with $X(0) = x^\ast$. To make the equation look simple, we abused the notation and denoted $e(t,X),\rho(t,X)$ as $e(t),\rho(t)$ so that 
$$
\frac{de}{dt} = -e(e-\rho).
$$
Since, $e=0$ is an unstable equilibrium point, $e(0)<0$ implies that $e(t)<0$ for all $t\geq 0$. Consequently,
$$
\frac{de}{dt} \leq -e^2.
$$  
Hence, 
$$
e\leq \frac{e(0)}{1+te(0)},
$$
which implies that 
$$
\lim_{t\to t_c^-} e(t) = -\infty, 
$$
for some $t_c\leq -(e(0))^{-1}$.
\end{proof}
The tools developed uptil this point will enable us to bound $u,\rho,u_x$. Observing the second assertion of Theorem \ref{local2}, to extend the local solutions to arbitrary time, it remains to control $\rho_x$. In view of this, taking the $x$ derivative of system \eqref{coupleerho}, and denoting $\xi := \rho_x$ and $\eta := e_x$, we obtain the system,
\begin{subequations}
\label{rhoxex}
\begin{align}
& \xi_t + f \xi_x = f_u(3\rho - 2e)\xi - f_u \rho\eta - f_{uu} \rho(\rho -e)^2, \label{rhoxex1} \\
& \eta_t + u \eta_x = e\xi +  (2\rho -3e)\eta. \label{rhoxex2}
\end{align}
\end{subequations}
Note that all the coefficients on the right hand side are bounded. We leverage this fact in following Lemma which will enable us to prove boundedness of $\xi$.
\begin{lemma}
\label{linearrhs}
Consider the following system of equations
\begin{align*}
& p_t + \mu_1 p_x = a p + bq + c,\\
& q_t + \mu_2 q_x = k p + l q + m,
\end{align*}
with initial conditions $p_0,q_0$. $a,b,c,k,l,m,p_0,q_0$ all belong to $L^\infty (\mathbb{R})$. Then 
$$
||p(t,\cdot)||_\infty ,||q(t,\cdot)||_\infty \leq \beta e^{\gamma t},
$$
where $\beta = 1+ ||p_0||_\infty + ||q_0||_\infty $ and\\ 
$$
\gamma = ||a||_\infty +||b||_\infty +||c||_\infty +||k||_\infty +||l||_\infty +||m||_\infty.
$$
\end{lemma}
Before proving this Lemma, we state the following corollary which is a result of a direct application of this Lemma to \eqref{rhoxex}.
\begin{corollary}
\label{rhoxbound1}
$\rho_x$ is bounded for all times with
$$
||\rho_x(t,\cdot)||_\infty \leq (1+||\rho_{0x}||_\infty+||e_{0x}||_\infty  ) e^{12(||f||_{C^2}+1)\max\{M^3,1\}t}.
$$
Here, $M$ is as in Proposition \ref{erhobound1}.
\end{corollary}
\begin{proof}
Comparing \eqref{rhoxex} to the equations in Lemma \ref{linearrhs} with $p = \xi$ and $q = \eta$, we obtain
\begin{align*}
& a = f_u(3\rho-2e),\quad b = -f_u\rho,\quad c = -f_{uu}\rho(\rho-e)^2,\\
& k = e,\quad l = 2\rho-3e, \quad m = 0.
\end{align*}
Consequently,
\begin{align*}
\gamma & = ||a||_\infty +  ||b||_\infty + ||c||_\infty +||k||_\infty +||l||_\infty +||m||_\infty\\
& \leq 5M||f||_{C^1}+ M||f||_{C^1} + 4 M^3||f||_{C^2} + M + 5M\\
& \leq 12(||f||_{C^2} + 1)\max\{ M^3,1\}.
\end{align*}
Hence, the result.
\end{proof}
Now we prove the Lemma.\\
\textit{Proof of Lemma \ref{linearrhs}:} We will show 
$$
p(t,\cdot)\leq \beta e^{\gamma t}.
$$
The other bounds can be shown in a very similar fashion. For this, define a function,
$$
G(t,x) := p - \beta e^{\delta\gamma t},
$$
with $\delta >1$ arbitrary but fixed constant. Note that $G(0,x)<0$ for all $x\in\mathbb{R}$. We will show that $G(t,\cdot)<0$ for all times. By way of contradiction, let $(t_0,x_0)$ be the first point such that $G(t_0,x_0) =0$. Hence, $G_t(t_0,x_0)\geq 0$ and $G_x (t_0,x_0) = p_x(t_0,x_0) = 0$. Using the $p$ equation,  
$$
\left. G_t\right|_{(t_0,x_0)} = \left. a \beta e^{\delta\gamma t} + b q + c - \delta \gamma \beta  e^{\delta \gamma t}\right|_{(t_0,x_0)}.
$$
Assuming without loss of generality that $||q(t,\cdot)||_\infty \leq \beta e^{\delta \gamma t}$. For otherwise, we can carry out the same procedure on $q$. Consequently,
\begin{align*}
\left. G_t\right|_{(t_0,x_0)} & \leq ||a||_\infty \beta e^{\delta\gamma t} + ||b||_\infty ||q||_\infty + \beta e^{\delta\gamma t} ||c||_\infty - \delta\gamma\beta e^{\delta\gamma t} \\
& \leq \gamma\beta e^{\delta \gamma t_0} - \delta \gamma \beta  e^{\delta \gamma t_0} < 0,
\end{align*}
which is a contradiction. Hence, $G(t,x)<0$ for all $t,x$. Note that all the calculations hold for all $\delta >1$. Taking limit as $\delta\to 1^+$, we prove the Lemma. \qed

Next, we prove Theorem \ref{thm2} using the tools developed above.\\
\textit{Proof of Theorem \ref{thm2}:} From Proposition \ref{ubound}, $u(t,\cdot)$ is uniformly bounded.
From Proposition \ref{erhobound1}, the condition $e_0(x):= u_{0x}(0,x)+\rho_0(0,x)\geq 0$ for all $x\in\mathbb{R}$ ensures that both $\rho(t,\cdot)$ and  $e(t, \cdot)=(u_x+\rho)(t,\cdot)$ are uniformly bounded. Using this, we have
\begin{align*}
 |u_x| & \leq |u_x + \rho| + \rho\\
 & \leq  2 \max\{\sup (u_{0x} + \rho_0), \sup\rho_0  \}.
\end{align*}   
Also, from Corollary \ref{rhoxbound1}, we have that $\rho_x$ is bounded for all times.
Therefore, from Theorem \ref{local2}, we have the existence of global-in-time solutions.

To prove the converse, we let $x^*\in\mathbb{R}$ with $e_0(x^*)<0$.
Then from Proposition \ref{erhobound1}, $e = u_x+\rho\to-\infty$ in finite time. Since $\rho(t,\cdot)\geq 0$, we conclude $u_x\to-\infty$ in finite time and $u$ loses $C^1$ smoothness.
This concludes the proof to Theorem \ref{thm2}. \qed

Now we will show that assuming the strict hyperbolicity of \eqref{matsys}, it can be more conveniently shown that $\rho_x$ is bounded. The key here is to use the Riemann invariant technique which was not available to us in the ansatz of the hypothesis of Theorem \ref{thm2}.
 More precisely, we have the following Proposition,
\begin{proposition}
\label{rhoxbound}
Let $f_u\leq 0$ and $|f(u_0(x))-u_0(x)|>0$ for all $x\in\mathbb{R}$. Given $e_0(x)\geq 0$ for all $x\in\mathbb{R}$ then $\rho_x$ is bounded for all time.
\end{proposition}
\begin{proof}
Note that from Lemma \ref{invreg}, $f(u(t,\cdot) - u(t,\cdot)<0$ for all further times and $R$ in turn, is well defined for any $t>0$.  Differentiating \eqref{R1} wrt $x$, we have
\begin{align*}
R_x & = \rho_x (f-u) e^{\int_{u_0}^u\frac{d\xi}{f(\xi)-\xi}} + \rho u_x f_u e^{\int_{u_0}^u\frac{d\xi}{f(\xi)-\xi}}\\
& = \rho _x (f-u) e^{\int_{u_0}^u\frac{d\xi}{f(\xi)-\xi}} + \rho (e-\rho) f_u e^{\int_{u_0}^u\frac{d\xi}{f(\xi)-\xi}} =: r.
\end{align*}
This leads to 
\begin{align}\label{rhox}
 \rho_x = \frac{r e^{-\int_{u_0}^u\frac{d\xi}{f(\xi)-\xi}}-f_u \rho (e-\rho)}{f-u}.
\end{align}
Hence, we conclude that $\rho_x$ is bounded if and only if $r$ is.
In view of this, differentiating \eqref{invsys1} w.r.t. $x$,  we obtain 
\begin{align*}
& r_t + fr_x + f_u u_x r = \rho r f_u+ \rho_x R f_u + \rho R f_{uu} u_x.
\end{align*}
Using $u_x= e-\rho$ and (\ref{rhox}) for $\rho_x$, we obtain 
\begin{align*}
 r_t + f r_x & = -f_u r(e-2\rho) + Rf_u \frac{r e^{-\int_{u_0}^u\frac{d\xi}{f(\xi)-\xi}}-f_u \rho (e-\rho)}{f-u} + \rho R f_{uu} (e-\rho)\\
& = -f_u r(e-3\rho) - \frac{R f_u^2 \rho(e-\rho)}{f-u} + \rho R f_{uu}(e-\rho)\\
& = -f_u r(e-3\rho) - \rho R (e-\rho) \left(\frac{f_u^2}{f-u} - f_{uu} \right). \numberthis \label{Rxeq}
\end{align*}
We used \eqref{R1} to obtain the second equation. 
Also, since $f_u(u)\leq 0$ for all $u$ in the domain of consideration, and 
$e_0(\cdot)\geq 0$, by Proposition \ref{erhobound1}, we have uniform bounds on $e,\rho$, and $R$.  Putting these facts together, we conclude that \eqref{Rxeq} is linear with coefficients bounded for any positive time. Therefore, $r$ is bounded for all times. 
Hence $\rho_x$ is bounded for all $x\in\mathbb{R}$, $t>0$.
\end{proof}

We can now prove the existence of classical solutions in the functional space as stated in Theorem \ref{thm1}.\\

\textit{Proof of Theorem \ref{thm1}:}
Putting together Propositions \ref{ubound}, \ref{erhobound1} and \ref{rhoxbound}, we obtain that if $e_0(x)\geq 0$ for all $x\in\mathbb{R}$, then $\rho(t,\cdot ),u(t,\cdot ),u_x(t,\cdot ),\rho_x(t,\cdot )$ remain bounded for all $t>0$.
For the only if part of Theorem \ref{thm1}, from \eqref{rhoe1b} we have that if $e(0)<0$ then $e = u_x+\rho\to-\infty$ in finite time. Therefore, $u_x\to-\infty$ in finite time and $u$ ceases to be in $C^1$. 
\qed
\begin{remark}
From Proposition \ref{rhobound}, we see that given the initial data does not cross $\{u:f(u)=u\}$, then even if the initial data is such that $e_0(x^\ast)<0$ for some $x^\ast$, there is no density concentration and solution breaks down through shock formation only. This is generic to a strictly hyperbolic system. However, this case is different from the usual strictly hyperbolic case. Unlike the usual situation where $\rho,u$ are both uniformly bounded, here, the particles may aggregate gradually. In other words, $\lim_{t\to\infty}\rho(t,x) = \infty$ when $e_0(x^\ast)<0$ for some $x^\ast$.  
\end{remark}

\section{Proof of Theorem \ref{thm3} ($f=f(\rho,u)$)}
\label{genf}
Let $ 0\leq\rho_0,u_0\in C^1_b(\mathbb{R})$. We begin by obtaining bounds on $u$ in terms of $\rho$. To do so, we show the existence of a curve $u=\phi(\rho)$ such that $f(\rho, \phi(\rho))\equiv\phi(\rho)$ for $\rho\in[0,M], M>0$. The result is stated in the following Lemma.
\begin{lemma}
\label{impfuncth}
If $\sup f_u(\rho, u) <1$ for $\rho \in [0, M]$, then there exists a unique continuous function, $\phi: [0, M]\longrightarrow\mathbb{R}$ such that $f(\rho, \phi(\rho)) \equiv \phi(\rho)$ for  $\rho \in [0, M]$. 
\end{lemma}
\begin{proof}
Set $h(\rho,u):=f(\rho,u)-u$. Let $\epsilon, a$ be constants such that  $\epsilon \leq f_u\leq  a<1$. Define the mapping $\Phi:C[0,M]\longrightarrow C[0,M] $ as follows,
$$
(\Phi \psi)(\rho) = \psi + \mu h(\rho,\psi).
$$
We claim that $\Phi$ is a contraction mapping for a nonzero choice of $\mu$. Then by fixed point theorem, we have the existence of a unique $\phi$ such that 
$$ 
h(\rho, \phi(\rho)) = 0\implies f(\rho, \phi(\rho)) = \phi (\rho),\quad \rho\in[0,M]. 
$$
To prove the above claim, we evaluate 
\begin{align*}
|\Phi \psi_1 - \Phi\psi_2| & = | (\psi_1 - \psi_2) +\mu (h(\rho,\psi_1)-h(\rho,\psi_2))|\\
& = |(1+ \mu h_u)(\psi_1 - \psi_2)|.
\end{align*}
Noting $\epsilon-1\leq h_u\leq a-1<0$.  Choosing $\mu = 1/(1-\epsilon)$, we have
\begin{align*}
|\Phi \psi_1 - \Phi \psi_2| & = \left( 1 + \frac{h_u}{1-\epsilon} \right)|\psi_1 - \psi_2| \\
& \leq \frac{a-\epsilon}{1-\epsilon}|\psi_1 - \psi_2|.
\end{align*}
Since $\epsilon \leq a<1$, $\Phi$ is a contraction. 

\end{proof}
Let $(\rho_0\geq 0,u_0)\in C^1(\mathbb{R})\times C^1(\mathbb{R})$, be bounded. We proceed with obtaining bounds on $u$ in terms of $\rho$, with results stated in the following Proposition.
\begin{proposition}
\label{uboundprop} Assume $\rho=\rho(t, x)$ is bounded, and $\sup f_u(\rho, u) <1$ for solutions under consideration. Then
\begin{align}
\label{ubound2}
& \min\{ \inf u_0 , \min \phi(\rho)\}\leq u(t,\cdot ) \leq \max\{ \sup u_0 , \max \phi(\rho)\},
\end{align}
for as long as $\rho$ exists. $\phi$ is as given in Lemma \ref{impfuncth}. 
\end{proposition}
\begin{proof} 
Along the characteristic path $(t,X)$ defined in \eqref{chpathu}, 
we have the following ODE,
\begin{align}
\label{eval2path}
& \frac{du}{dt} = \rho(t,X)\bigg( f(\rho(t,X),u(t,X)) - u(t,X) \bigg).
\end{align}
This by Lemma \ref{impfuncth} can be rewritten as 
$$
\frac{du}{dt} = \xi(t) (u(t,X)-\phi(\rho(t, X)),
$$
where $\xi:= \rho(t,X)( f_u(\rho(t,X),\tilde u(t,X)) -1)$ with $\tilde u$ being  
a mean value between $u$ and $\phi(\rho)$.  By the assumption we have $\xi \leq 0$, hence we have 
$$
\xi(t) (u(t,X)- \min \phi(\rho))\leq \frac{du}{dt} \leq  \xi(t) (u(t,X)- \max \phi(\rho)).
$$  
Therefore by comparison, we have on the path, $(t,X)$, with any fixed point, $(0,x_0)$, both 
$$
u(t,X) \leq \max\{ u_0(x_0), \max \phi(\rho)\}, 
$$
and 
$$
u(t,X) \geq  \min \{ u_0(x_0), \min \phi(\rho)\},
$$
as long as $\rho$ exists.
 Combining all characteristics, we obtain the bounds as stated in \eqref{ubound2}. 
\end{proof}

By using the technique as in Proposition \ref{erhobound1}, we will show uniform bounds on $\rho$ and $u_x$ by a coupling between $\rho$ and $e(t,x):= u_x(t,x) + \rho(t,x)$. The result is stated 
in the following Proposition.
\begin{proposition}
\label{erhobound2}
Let $f_u(\rho, u)\leq 0$ for solutions under consideration. If $e_0(x),\rho_0(x)\in [0,M]$ for all $x\in\mathbb{R}$, then $e(t,\cdot),\rho(t,\cdot )\in [0,M]$ for all $t>0$. 
Moreover, if $e_0(x^\ast ) <0 $ for some $x^\ast$, then $\exists x_c, t_c>0$ such that $\lim_{t\to t_c^-} e(t,x_c) = -\infty$. 
\end{proposition}
\begin{proof} We define the characteristic path, $(t,Y)$ corresponding to $\lambda_1$ with $Y$ satisfying the ODE,
\begin{align}
\label{chpathfgen}
& \frac{dY}{dt} = \rho (t,Y)f_\rho (\rho(t,Y),u(t,Y)) + f(\rho(t,Y),u(t,Y)),\quad Y(0) = y_0,
\end{align}
and recall the  characteristic path $(t, X)$, we get the following system,
\begin{subequations}
\label{rhoe2}
\begin{align}
& \frac{d}{dt}\rho(t,Y) = f_u \Big(\rho(t,Y),u(t,Y)\Big) \rho(t,Y)(\rho(t,Y) - e(t,Y)), \label{rhoe2a}\\
& \frac{d}{dt}e(t,X) = -e(t,X)(e(t,X)-\rho(t,X)). \label{rhoe2b}
\end{align}	
\end{subequations}
Again we set $e_0(x):=u_{0x}(x) + \rho_0(x)$ as initial data for $e$. With the condition $f_u(
\rho, u) \leq 0$, the proof of the claimed results is entirely similar to that in the proof of Proposition
\ref{erhobound1}.
\end{proof}
Next, we will show a sufficient condition to bound $\rho_x$. In view of this, taking derivative of transport equations for $\rho,e$ and setting $\xi:=\rho_x$ and $\eta:=e_x$, we obtain,
\begin{subequations}
\label{rhoxexgen}
\begin{align}
& \xi_t + \lambda_1 \xi_x = -(\rho f)_{\rho\rho}\xi^2 + \Big( 2\rho f_{\rho u}(\rho-e) + f_u(3\rho -2e) \Big)\xi \label{rhoxexgen1}  \\
& \qquad\qquad \qquad - \rho f_u\eta - f_{uu}\rho(\rho-e)^2, \nonumber \\
& \eta_t + u\eta_x = e\xi + (2\rho-3e)\eta. \label{rhoxexgen2}
\end{align}
\end{subequations}
We state a short Lemma which is the first step towards proving bounds on $\rho_x$.
\begin{lemma}
\label{etabound}
Suppose $e\geq 0$. As long as $\xi$ exists:
\begin{enumerate}
    \item 
    if $\eta(0,x)\geq 0$ and $\xi(t,x)\geq 0$ for $t\geq 0, x\in\mathbb{R}$, then $\eta(t,x)\geq 0$ for $t\geq 0, x\in\mathbb{R}$,
    \item 
    if $\eta(0,x)\leq 0$ and $\xi(t,x)\leq 0$ for $t\geq 0, x\in\mathbb{R}$, then $\eta(t,x)\leq 0$ for $t\geq 0, x\in\mathbb{R}$.
\end{enumerate}
\end{lemma}
\begin{proof}
If $\xi$ exists then \eqref{rhoxexgen2} can be reduced to ODE along the path $(t,X)$,
$$
\frac{d\eta}{dt} = e\xi(t,X) + (2\rho-3e)\eta(t,X). 
$$
Upon integration, we obtain
$$
\eta(t,X) = \eta(0,x_0)e^{\int (2\rho - 3e)\, d\tau} + \int_0^t e^{\int_s^t (2\rho - 3e) \, d\tau} e\xi\, ds.
$$
From this expression, one can conclude that if $\eta(0,x_0)\geq 0$ and $\xi(t,x)\geq 0$ for $t\geq 0, x\in\mathbb{R}$, then $\eta(t,x)\geq 0$. The second assertion can be concluded similarly. 
\end{proof}
Lemma \ref{etabound} will help us to prove the following Proposition.
\begin{proposition}
\label{xibound}
(Bounds on $\xi$) Assume $e(0,x)\geq 0$ $\forall x\in\mathbb{R}$.
\begin{enumerate}
    \item Suppose $\eta(0,x),\xi(0,x) \geq 0$ for all $x\in\mathbb{R}$. Also assume $(\rho f)_{\rho\rho}\geq 0$ and $f_{uu}\leq 0$ for solutions under consideration. Then $\eta(t,x)\geq 0$, $\xi(t,x)\geq 0$, for $t>0, x\in\mathbb{R}$ and for any $t>0$, $\xi(t,\cdot)$ has a uniform upper bound. 
    \item Suppose $\eta(0,x),\xi(0,x) \leq 0$ for all $x\in\mathbb{R}$. Also assume $(\rho f)_{\rho\rho}\leq 0$ and $f_{uu}\geq 0$ for solutions under consideration. Then $\eta(t,x)\leq 0$, $\xi(t,x)\leq 0$, for $t>0, x\in\mathbb{R}$ and for any $t>0$, $\xi(t,\cdot)$ has a uniform lower bound. 
\end{enumerate}
\end{proposition}
\begin{proof}
We will prove the first assertion. The second can be proved similarly. Let $\epsilon>0$ be fixed. We add this barrier of $\epsilon$ to $\xi(0,x),\eta(0,x)$ and take the resulting sum to be the new initial profiles. Consequently, we have $\xi(0,\cdot)\geq\epsilon$ and $\eta(0,\cdot)\geq\epsilon$. Note that if we are able to show $\xi(t,\cdot)\geq 0$ for all further times along with $\xi$ being upper bounded, then an application of Lemma \ref{etabound} and then letting $\epsilon\to 0^+$ proves the Proposition. 

For this, we will use \eqref{rhoxexgen1}. Note that $(\rho f)_{\rho\rho}\geq 0$ implies the coefficient of $\xi^2$ in right hand side of \eqref{rhoxexgen1} is nonpositive and hence, $\xi$ is bounded from above as long as the all the other terms on right hand side are bounded. 
Next, owing to the assumptions, from \eqref{rhoxexgen1} we have,
\begin{align*}
\xi_t + \lambda_1\xi_x & \geq -(\rho f)_{\rho\rho} \xi^2 + \Big( 2\rho f_{\rho u}(\rho-e) + f_u(3\rho -2e) \Big)\xi\\
& = \xi \Big( -(\rho f)_{\rho\rho}\xi + 2\rho f_{\rho u}(\rho-e) + f_u(3\rho -2e) \Big).
\end{align*}
Therefore, $\xi(t,\cdot)\geq 0$. Above calculations hold for all $\epsilon>0$. Letting $\epsilon\to 0^+$ gives the result.
\end{proof}

Using the tools developed in this section, we finally prove the main Theorem.

\textit{Proof of Theorem \ref{thm3}:} From Proposition \ref{uboundprop}, we conclude that $u$ is uniformly bounded if $\rho$ is. Based on the assumption on initial data we have 
 $e_0(x)\geq 0$ for all $x\in\mathbb{R}$, then from Proposition \ref{erhobound2}, we have that $\rho$ is uniformly bounded. Along with that, we also obtain uniform bounds on $u_x$ since $e$ and $\rho$ are uniformly bounded. This proves the first two assertions of the Theorem. 
 
For the third assertion (global solution), assuming its hypothesis, we obtain that from Proposition \ref{xibound}, $\rho_x$ is bounded for all times. Hence, from Theorem \ref{local2}, we obtain the global existence of solutions. 

Assertion 4 (Finite time breakdown): If $e_0(x^*)<0$ for some $x^*$. From \eqref{rhoe2b}, along the characteristic path, $(t,X)$, with $X(0) = x^*$, we have that $e(0)<0$. Then from Proposition \ref{erhobound2}, $e = u_x+\rho\to-\infty$ as $t \to t_c$, a finite time. Since $\rho(t,\cdot)\geq 0$, we conclude $u_x\to-\infty$ in finite time.  \qed

\appendix
\section{Proof of Theorem \ref{local2}}
We begin with an auxiliary system of form 
\begin{align}\label{UU}
U_t +\Lambda U_x= F(U),
\end{align}
with $f=f(n, v)$
$$
U= \left[
\begin{array}{c}
n\\
v\\
q
\end{array}
\right], \quad \Lambda=\text{diag}(nf_n+f, v, v), \quad F= \left[
\begin{array}{c}
f_v n (n -q )\\
n(f-v)\\
q(n - q)
\end{array}
\right].
$$
This is a symmetric hyperbolic system. Given $U(0,\cdot)\in C^1_b(\mathbb{R})$, \cite[Theorem 7.7.1]{Da16} gives the local existence of classical solutions to \eqref{UU}, $U$ in $C^1([0,T)\times \mathbb{R})$ for some $T \in(0, \infty]$. The interval $[0, T)$ is maximal in the sense that if $T<\infty$, then as $t \uparrow T$, 
$\|\partial_x U(t, \cdot)\|_\infty \to \infty$ and the range of $U$ becomes unbounded. 
We claim that if $q(0,x)=n(0,x)+v_x(0,x)$, then  
\begin{align}\label{qnv}
q(t, x)=n(t, x)+v_x(t, x) \quad t\in (0, T). 
\end{align}
Substitution of this relation into the first two equations in the auxiliary system, we obtain 
$$
n_t +(fn)_x=0, \quad v_t+vv_x=n(f-v). 
$$
This is the same system as \eqref{matsysgen}.  We thus take $n(x, 0)=\rho_0(x), v(x, 0)=u_0(x)$, and obtain the local solution by setting 
$$
(\rho, u)=(n, v)\in C^1([0,T); \mathbb{R}). 
$$
We now return to the proof of (\ref{qnv}). 
Using the $(n,q,v)$ system,
\begin{align*}
(v_x+n-q)_t + v(v_x+n -q)_x & = v_{xt} + vv_{xx} +n_t + vn_x - q_t - vq_x\\
& =  -v_x^2 + n_x(f-v)+ n(f-v)_x+ (v-f-n f_n)n_x \\
& \quad + f_v n(n-q) -q(n -q)\\
& = -v_x^2 + n(f_v-1)v_x + f_v n(n-q) - q(n-q)\\
& = -v_x(v_x+n) + f_v n(v_x+n-q)-q(n-q)\\
& = -v_x(v_x+n-q) + f_v n(v_x+n-q)-q(n+v_x-q)\\
& = -(v_x-n f_v +q)(v_x+n -q).
\end{align*}
Since $\left. (v_x+n-q)\right|_{t=0} = 0$, we have that $q=n+v_x$ for any $t>0$. 

Therefore, for \eqref{matsysgen}, we also have $\rho+u_x$ in $C^1((0,T)\times\mathbb{R})$.

Furthermore, using the 3rd equation and $q=n+v_x$ we obtain the following,  
$$
q_{xt} +vq_{xx}=(2n-3q)q_x +qn_x.  
$$
This implies that for bounded $(n, q)$, $q_x$ is bounded if $n_x$ is bounded.  Returning to the $(\rho, u)$ variables, we see that the local solution can be extended beyond $T$ if  $\partial_x(\rho, u)$ is bounded and $(\rho, u)$ is bounded at time $T$.

\section*{Acknowledgments}
This research was supported by the National Science Foundation under Grant DMS1812666.

\bigskip

\bibliographystyle{abbrv}

\end{document}